\documentclass[11pt]{amsart}
\usepackage{amsmath} 
\usepackage{amsfonts}  
\usepackage{amssymb}
\usepackage{enumerate}
\usepackage{tikz}
\usepackage{tikz-cd}
\usepackage{graphicx}

\newtheorem{theorem}{Theorem}[section]
\newtheorem{cor}[theorem]{Corollary}
\newtheorem{prop}[theorem]{Proposition}
\newtheorem{lemma}[theorem]{Lemma}

\newtheorem{defn}[theorem]{Definition}

\theoremstyle{definition}
\newtheorem{example}[theorem]{Example}
\newtheorem{remark}[theorem]{Remark}

\newcommand{\N}{\mathbb{N}}

\newcommand{\G}{\mathbb{G}}

\newcommand{\PP}{\mathbb{P}}
\newcommand{\TT}{\mathbb{T}}
\newcommand{\cL}{\mathcal{L}}
\newcommand{\cA}{\mathcal{A}}
\newcommand{\cB}{\mathcal{B}}

\newcommand{\Sphere}{\mathbb{S}}
\newcommand{\Hom}{\operatorname{Hom}}
\newcommand{\leftlin}{LL}

\newcommand{\transfer}{\operatorname{tr}}
\newcommand{\wreath}{\!\wr\!}
\newcommand{\Mod}{\operatorname{Mod}}

\newcommand{\sm}{\wedge}

\begin{document}

\title[Morava $E$-theory of wreath products]{On the Morava $E$-theory of wreath products of symmetric groups}
\author{Peter Nelson} 
\date{\today}

\begin{abstract}
Strickland gave an interpretation of a quotient of the Morava $E$-theory of symmetric groups in terms of algebraic geometry. We identify an analogous quotient  of the Morava $E$-theory of wreath products of symmetric groups with a tensor product of Strickland's quotients. This allows an algebro-geometric interpretation similar to Strickland's.
\end{abstract}

\maketitle

\section{Introduction}

In \cite{strickland_symmetric}, Strickland provides an algebro-geometric interpretation of a certain quotient of the Morava $E$-theory of symmetric groups. Namely, let $E$ be a Morava $E$-theory of height $h$ with formal group $\G$, $\Sigma_m$ the symmetric group on $m$ letters, and $I_m$ the ideal in $E^0(B\Sigma_m)$ generated by the images of the stable transfer maps $E^0(B(\Sigma_i\times \Sigma_{m-i})) \to E^0(B\Sigma_m) $ for $0<i<m$. Strickland then proves the following.

\begin{theorem}[\cite{strickland_symmetric}, Theorem 1.1]
The ring $E^0(B\Sigma_m)/I_m$ classifies subgroup schemes of $\G$ of order $m$.
\end{theorem}

Rezk later (\cite{congruence_criterion}, Section 11) provided an interpretation of these formal schemes as morphisms ``of degree m'' in a certain category scheme. In particular, the rings $E^0(B\Sigma_m)/I_m$ come with a pair of ring maps $s,t: E_0 \to E^0(B\Sigma_m)/I_m$ classifying the source and target of a morphism. The map $s$ comes from the standard $E_0$-algebra structure on cohomology, and the map $t$ comes from a certain power operation. This gives $E^0(B\Sigma_m)/I_m$ an $E_0$-$E_0$-bimodule structure, denoted as  
\begin{equation*}
\biggl.^t\!\left(\frac{E^0(B\Sigma_m)}{I_m}\right)\!\biggr.^s,
\end{equation*}
with $s$ regarded as defining a right $E_0$-module structure and $t$ regarded as defining a left module structure.

We then have that the tensor product 
\begin{equation*}
\biggl.^t\!\left(\frac{E^0(B\Sigma_m)}{I_m}\right)\!\biggr.^s \bigotimes \biggl.^t\!\left(\frac{E^0(B\Sigma_n)}{I_n}\right)\!\biggr.^s
\end{equation*}
classifies a composable pair of morphisms (of degree $m$ and $n$, respectively) in Rezk's category scheme.

One may wonder if this tensor product can be realized as an object coming from the Morava $E$-theory of certain groups. Contemplation of the power operation defining the left module structure on $E^0(B\Sigma_m)/I_m$ indicates an analogous quotient of the $E$-cohomology of a wreath product of symmetric groups as a likely candidate. We prove that this is in fact the case. 

\begin{theorem}\label{main_theorem_1} There is an ideal $I_{m,n}$ in $E^0(B\Sigma_m\wr\Sigma_n)$ analogous to Strickland's ideals. The quotient $E^0(B\Sigma_m\wr\Sigma_n)/I_{m,n}$ has an $E_0$-module structure $t$ in addition to the usual module structure $s$. There is then an isomorphism of bimodules

\[\biggl.^t\!\left(\frac{E^0(B\Sigma_m\wreath\Sigma_n)}{I_{m,n}}\right)\!\biggr.^s \cong \biggl.^t\!\left(\frac{E^0(B\Sigma_m)}{I_m}\right)\!\biggr.^s \bigotimes \biggl.^t\!\left(\frac{E^0(B\Sigma_n)}{I_n}\right)\!\biggr.^s.\]
Further, both sides are finite free right $E_0$-modules.
\end{theorem}

In fact, we can extend this to iterated wreath products of symmetric groups to get our main theorem.

\begin{theorem}\label{main_theorem_2}
Let $\mathbf{m} = (m_1, \ldots, m_\ell)$ be an $\ell$-tuple of positive integers, and let $\Sigma_\mathbf{m} = \Sigma_{m_1}\wr\ldots\wr\Sigma_{m_\ell}$ be an iterated wreath product of symmetric groups. There are ideals $I_\mathbf{m}$ in $E^0(\Sigma_\mathbf{m})$ analogous to  Strickland's ideals. The quotient $E^0(\Sigma_\mathbf{m})/I_\mathbf{m}$ has an $E_0$-module structure $t$ in addition to the usual module structure $s$.
Then there is an isomorphism of $E_0$-$E_0$ bimodules
\[\biggl.^t\!\left(\frac{E^0(B\Sigma_{m_1} \wreath \ldots \wreath \Sigma_{m_\ell})}{I_{\mathbf{m}}}\right)\!\biggr.^s \cong \biggl.^t\!\left(\frac{E^0(B\Sigma_{m_1})}{I_{m_1}}\right)\!\biggr.^s \bigotimes \cdots \bigotimes \biggl.^t\!\left(\frac{E^0(B\Sigma_{m_\ell})}{I_{m_\ell}}\right)\!\biggr.^s.\]
Further, both sides are finite free right $E_0$-modules.
\end{theorem}

We prove this as Theorem \ref{main_theorem}. Here is a brief sketch of the ideas involved in the isomorphism. First, we take $E_0$-linear duals, and prove an isomorphism there. This is accomplished by identifying the duals of the objects as the outcome of a ``linearization'' process applied to certain functors of Rezk, related to Morava $E$-theory power operations. We then show we can dualize back and recover the intended isomorphism.

This paper is organized as follows. Section \ref{prelims} covers several useful preliminary notions, namely extended power constructions, transfer maps, power operations and bimodule duality. This section also collects several useful propositions about these notions from the literature, to have them at our disposal.

 Section \ref{cohom_wreath_products} collects some facts about the Morava $E$-theory of groups. It also introduces the ideals $I_\mathbf{m}$ and the quotients $E^0(B\Sigma_\mathbf{m})/I_\mathbf{m}$ from the statement of the main theorem. 

Section \ref{the_functors_TT} recalls some functors related to Morava $E$-theory power operations, originally defined by Rezk in \cite{congruence_criterion}. Section \ref{left_linearization_section} defines a notion of left linearization of functors following Johnson-McCarthy \cite{johnson_mccarthy} and Rezk \cite{power_ops_koszul}, and proves a ``chain rule'' for left linearization. Section \ref{bimodules} introduces the objects that are putatively dual to the quotients $E^0(B\Sigma_\mathbf{m})/I_\mathbf{m}$, and identifies them as certain values of the left linearization of Rezk's functors. Finally, Section \ref{proof_of_main} finishes the proof of the main theorem.

\subsection*{Acknowledgments}
The bulk of this work was carried out while the author was a graduate student under the supervision of Charles Rezk. We would like to thank him for his guidance and his blessing to pursue this problem, as well as the partial support from his NSF grant DMS-1406121. The remainder of the work was carried out at the University of Haifa, and we would also like to acknowledge the support of Israel Science Foundation grant 1138/16.

\section{Preliminaries}\label{prelims}

\subsection{Notation and Conventions}
If $E$ is a spectrum, we will generally write $E_*$ for $\pi_* E$, and $E_0$ for $\pi_0 E$. 

If $X$ is a space, we use $\Delta$ to refer to both the diagonal map of spaces $X \to X\times X$  and the induced map
$\Sigma_+^\infty X \to \Sigma_+^\infty X \wedge \Sigma_+^\infty X$.

There are at least two different conventions for wreath products. Both would be reasonable in the context of this paper, so we take a moment to be clear. 
\begin{defn}
Let $G\leq\Sigma_m$ be a subgroup of a symmetric group, and $H$ be any finite group. The wreath product $H\wr G$ is the semidirect product which sits in the following short exact sequence.
\begin{equation*}
	1 \to H^m \to H\wr G \to G \to 1
\end{equation*}
The action of $G$ on the product $H^m$ comes from permuting coordinates via the inclusion $G\leq \Sigma_m$.
\end{defn}

All groups appearing in this paper will implicitly (or explicitly) come with a specified embedding into a symmetric group. We make special use of products and wreath products, so here are the embeddings we have in mind for those.

\begin{example}\label{embeddings}
Let $X_m$ be an ordered set with $m$ elements, and let $\Sigma_m$ act on $X_m$ by reordering the elements. If $i+j =m$, we can embed $\Sigma_i$ into $\Sigma_m$ as those elements which permute just the first $i$ elements of $X_m$. Similarly we can embed $\Sigma_j$ as those elements which permute just the last $j$ elements. Together these give our embedding $\Sigma_i \times \Sigma_j \leq \Sigma_{i+j}= \Sigma_m.$ Of course, if $G\leq \Sigma_i$ and $H\leq \Sigma_j$ we can compose inclusions $G\times H \leq \Sigma_i\times \Sigma_j \leq \Sigma_{i+j}$.

For wreath products, we consider $\Sigma_m\wr\Sigma_n$ as a subgroup of $\Sigma_{mn}$ as follows. Divide $X_{mn}$ into $n$ blocks of $m$ letters each. Then we can make $\Sigma_m^n$ act on $X_{mn}$ by having each factor act on one of the blocks separately. We can also have $\Sigma_n$ act on $X_{mn}$ by permuting the blocks, but leaving order alone within each block. Combining these give an action of $\Sigma_m\wr\Sigma_n$ on $X_{mn}$ and so an embedding $\Sigma_m\wr\Sigma_n$ into $\Sigma_{mn}$. As before, if $H \leq \Sigma_m$ and $G\leq \Sigma_n$ we can compose to get an inclusion $H\wr G \leq \Sigma_m\wr\Sigma_n \leq \Sigma_{mn}$
\end{example}

\subsection{Morava $E$-theory}
Let $\kappa$ be a perfect field of characteristic $p>0$ and $\Gamma$ a formal group of height $h< \infty$ over $\kappa$. In this situation, there is a notion of deformation of $\Gamma$. Lubin and Tate \cite{Lubin_Tate} showed that there exists a universal such deformation defined over a complete local ring $R_{\operatorname{LT}}$ with residue field $\kappa$, and Morava observed that there is an even-periodic homotopy-commutative ring spectrum $E$ with $\pi_0(E) \simeq R_{\operatorname{LT}}$ whose associated formal group is the universal deformation of $\Gamma$. We call this Morava $E$-theory; it is also sometimes called the Lubin-Tate spectrum. Other authors often use the letter $n$ to denote height but we will need several indexing variables and instead use the letter $h$ for height.

Goerss, Hopkins and Miller \cite{goerss_hopkins} showed that $E$ admits an essentially unique $E_\infty$ ring structure. For more on Morava $E$-theory, see \cite{charles_hopkins_miller}. As is customary, we supress all dependence of Morava $E$-theory on anything but the prime $p$ and the height $h$. Both $p$ and $h$ will be fixed throughout this paper, so we generally supress those as well, and refer to any of these $E_\infty$ ring spectra as $E$. Unless specifically stated otherwise, the letter $E$ will only mean this.

There is a related cohomology theory, called Morava $K$-theory and generally denoted $K(h)$. We make use of Bousfield localization with respect to $K(h)$, and denote it as $L_{K(h)}$. For more, see \cite{hovey_strickland}. We also make use of a variant of homology using this localization.

\begin{defn}\label{completed_homology}
If $X$ is a spectrum (or a space) the completed $E$-homology $E^{\wedge}_\ast(X)$ is defined by the formula
\[ E^{\wedge}_\ast(X) = \pi_\ast L_{K(h)}(E\wedge X).\]
\end{defn}

We also use $E^{\wedge}_0(X)$ for $\pi_0$ of the same localization. This does not quite define a homology theory, but is often more convenient than the unlocalized version.

\subsection{Extended powers}\label{extended_powers_section}
Let $E$ be any $E_\infty$ ring spectrum. We will be thinking of a Morava $E$-theory spectrum, though nothing in this subsection will need anything specific about that case. We cite \cite{h_infinity_book} several times in this section. Though the statements there are only for $E$ the sphere spectrum, modern symmetric monoidal categories of spectra (for instance, that of \cite{EKMM}) allow their proofs to go through verbatim.

\begin{defn}
Let $M$ be an $E$-module. The $m$-\emph{th extended  power} of $M$ is the homotopy orbits $(M^{\wedge_E m})_{h\Sigma_m}$ of the $m$-th smash power of $M$ under the action of $\Sigma_m$ which permutes the factors. We will call this $\PP^E_m(M)$, or just $\PP_m(M)$ if $E$ is clear. If $G \leq \Sigma_m$, we can take the $G$-homotopy orbits instead, and we will call this functor $\PP^E_G$.
\end{defn}

Note that $\PP^E_G$ may depend on the embedding $G \leq \Sigma_m$. 

The same formula defines a space-level construction, which we will also denote by $\PP_G$. 

Here are a few facts about extended powers. 

\begin{prop}[\cite{h_infinity_book} I.2.2]
Let  $X$  be an unpointed space and $G\leq \Sigma_m$. Then there is an isomorphism $\Sigma_+^\infty(\PP_G(X)) \simeq \PP_G^\mathbb{S}(\Sigma_+^\infty(X))$
\end{prop}

\begin{prop}\label{extended_power_of_unit}
$\PP_m^E(E) \cong E \wedge \Sigma_+^\infty B\Sigma_m$.
\end{prop}

\begin{proof}
This follows from the definition of $\PP_m^E$ and the fact that $E^{\wedge_E m}$ is isomorphic to $E$.
\end{proof}

Let $\Lambda_{m,k}$ be the set of $(k+1)$-tuples of nonnegative integers whose sum is $m$: 
\[
\Lambda_{m,k} = \left\{ (\lambda_0, \ldots, \lambda_k) \in \N^k \mid \sum_{i=0}^k \lambda_i = m \right\}.
\]

\begin{prop}[\cite{h_infinity_book} II.1.1]\label{extended_power_of_wedge} There is a natural equivalence
\[\PP^E_m(M \vee N) \xleftarrow{\sim} \bigvee_{i+j = m}\PP^E_i(M) \wedge_E \PP^E_j(N).\]

Inductively this gives natural equivalences
\[\PP^E_m\left(\bigvee_{i=0}^k M_i\right) \xleftarrow{\sim} \bigvee_{\lambda\in \Lambda_{m,k}} \bigwedge_{i=0}^k \PP_{\lambda_i}^E(M_i),\]
where the smash product on the right hand side is taken in the category of $E$-modules
\end{prop}

\begin{prop} \label{comp_P}
If $H\leq \Sigma_m$ and $G \leq\Sigma_n$, the inclusions of Example \ref{embeddings} give isomorphisms 
\[\PP^E_G \circ \PP^E_H \simeq \PP^E_{H\wr G}.\]
and 
\[\PP^E_{G\times H} \simeq \PP^E_G \wedge_E \PP^E_H.\]
\end{prop}

\begin{proof}
The first statement is a special case of the isomorphism $\beta$ on page 5 of \cite{h_infinity_book}. The second is a special case of the isomorphism $\alpha$ on the same page.
\end{proof}

The following proposition is helpful when considering the interaction between localization and extended powers.

\begin{prop}[\cite{congruence_criterion} 3.16]\label{power_of_localization}
The functors $\PP_m$ preserve $K(h)$-homology isomorphisms. In particular, there is a natural isomorphism $L_{K(h)}\PP_m \to L_{K(h)}\PP_m L_{K(h)}$.
\end{prop}

Our last fact about extended powers is crucial to the whole theory of power operations for Morava $E$-theory. We emphasize that it is not formal; it uses some very specific things about Morava $E$-theory. We first recall the following definition.

\begin{defn}
An $E$-module $M$ is a \emph{finite free} $E$-module if it is a finite wedge of suspensions of $E$. Equivalently, $M$ is a finite free $E$-module if $\pi_*M$ is a finite free $E_*$-module.
\end{defn}

\begin{prop}[\cite{congruence_criterion}, Proposition 3.12]\label{finite_free_equivalence}
Taking homotopy groups $\pi_*$ is an equivalence of categories between the homotopy category of finite free $E$-modules and the category of finite free $E_*$-modules. Similarly, $\pi_0$ is an equivalence between the homotopy category of finite free $E$-modules concentrated in even degrees and the category of finite free $E_0$-modules.
\end{prop}

\begin{prop}[\cite{congruence_criterion} Proposition 3.17]
If $M$ is a finite free $E$-module, then so is  $L_{K(h)}\PP_m(M)$. 
\end{prop}

\subsection{Transfer maps}\label{transfer_section}

This section will give a brief refresher on the notion of transfer maps, and state a few properties that will be useful later. For more details, see Chapter 4 of \cite{adams_infinite_loops}.

\begin{defn}
For $f:Y \to X$ a finite sheeted covering map of spaces, there is a ``wrong-way'' map of spectra $f^!: \Sigma_+^\infty X \to \Sigma_+^\infty Y$ called the ``transfer''. It is functorial in covering maps,  takes pullback squares to commutative squares, and takes products to smash products.
\end{defn}

Our main examples will be classifying spaces of finite groups, and somewhat more generally, extended powers of spaces. If $H\leq G\leq \Sigma_k$, and $X$ is any space, then the natural map $\PP_H(X) \to \PP_G(X)$ is a finite covering map (with fiber $G/H$), and we will label the associated transfer $\transfer^G_H$. Taking $X= *$, we get a map between suspension spectra of classifying spaces. 

For $\Sigma_i \times \Sigma_{m-i} \leq \Sigma_m$, there is another description of the transfer maps which will be used later. 

\begin{prop}[\cite{h_infinity_book}, II.1.5]\label{transfer_is_diagonal}
The map
\[\transfer_{\Sigma_i \times \Sigma_{m-i}}^{\Sigma_m}:\PP_m(X) \to \PP_{\Sigma_i \times \Sigma_{m-i}}(X) \simeq \PP_i(X) \wedge \PP_{m-i}(X)\]
 is equal to the composite 
\begin{equation*}
	\begin{tikzcd}
		\PP_m(X) \arrow{r}{\PP_m(\Delta)} & \PP_m(X\vee X) \arrow{r}{\operatorname{proj}} & \PP_i(X) \wedge \PP_{m-i}(X),
	\end{tikzcd}
\end{equation*}
where $\operatorname{proj}$ is the projection onto the appropriate factor from the equivalence in Lemma \ref{extended_power_of_wedge}.
\end{prop}

\begin{prop}\label{transfer_ext_powers}
Transfers commute with extended powers in the sense that the following diagram commutes for $H$ a subgroup of $G$.
\begin{equation*}{}
	\begin{tikzcd}
		\PP_m(\Sigma_+^\infty BG) \arrow{r}{\PP_m(\transfer^G_H)} & \PP_m(\Sigma_+^\infty BH)\\
		\Sigma_+^\infty BG\wreath\Sigma_m \arrow{u}{\simeq} \arrow{r}{\transfer^{G\wr\Sigma_m}_{H\wr\Sigma_m}} &\Sigma_+^\infty BH\wreath\Sigma_m \arrow{u}{\simeq}
	\end{tikzcd}
\end{equation*}
\end{prop}

We couldn't quite find this statement explicitly in the literature, so we at least give an indication of how to assemble a proof from existing texts. 

\begin{proof}
One construction of the transfer $\Sigma_+^\infty BG \to \Sigma_+^\infty BH$ is given in \cite{lewis_may_steinberger} as homotopy fixed points of a (naively) $G$-equivariant map $\tau: \Sphere \to \Sigma_+^\infty (G/H)$ called the ``pretransfer.'' From this construction, the claim follows from assembling Lemmas IV.2.6 and IV.2.7 from \cite{lewis_may_steinberger} (with all universes trivial and $\alpha$ the inclusion $H\hookrightarrow G$) as well as the isomorphism $\beta$ from page 5 of \cite{h_infinity_book}.
\end{proof}

\subsection{The Total Power Operations}\label{power_ops}

 Let $X$ be a spectrum. The $E_\infty$ structure on Morava $E$-theory allow us to define a ``total power operation'' on $E^0(X)$ as follows. Recall that an $E_\infty$ structure on $E$ is the structure of an algebra on $E$ for the monad $\PP^\Sphere$. In particular, there are structure maps $\mu_m: \PP^\Sphere_mE \to E$. For any spectrum map $\alpha: X \to E$, we can apply the extended power $\PP^\Sphere_m$ and compose with $\mu_m$ to get a composite map
 \[ P_m(\alpha): \PP^\Sphere_m(X) \xrightarrow{\PP^\Sphere_m(\alpha)} \PP^\Sphere_m(E) \xrightarrow{\mu_m} E.
 \]
 This composite gives a cohomology class in $E^0(\PP_m^\Sphere(X)),$ and so we obtain a \emph{set} map 
 \[E^0(X) \xrightarrow{P_m} E^0(\PP_m^\Sphere X).
 \]
This map is multiplicative, but \emph{not additive}. However, we have the following formula for total power operations applied to a sum.
\begin{prop}[\cite{h_infinity_book}, II.2.1]\label{total_power_of_sum}
If $\alpha$ and $\beta$ are cohomology classes in $E^0(X)$, then 
\[P_m(\alpha + \beta) = \sum_{i+j=m} \transfer_{\Sigma_i \times \Sigma_j}^{\Sigma_m}(P_i(\alpha) \boxtimes P_j(\beta)),
\]
where $\boxtimes$ is the external product in $E$-cohomology.
\end{prop}

The total power operations also interact nicely with transfers from subgroups in another way.
\begin{prop}\label{transfers_and_total_power_ops}
If $G$ is a finite group, and $H$ is a subgroup of $G$, then the diagram
\begin{equation*}
	\begin{tikzcd}
		E^0(BH) \arrow{r}{P_m} \arrow{d}{\transfer_H^G} & E^0(BH\wreath\Sigma_m) \arrow{d}{\transfer_{H\wr\Sigma_m}^{G\wr\Sigma_m}}\\
		E^0(BG) \arrow{r}{P_m} & E^0(BG\wreath\Sigma_m)
	\end{tikzcd}
\end{equation*}
commutes.
\end{prop}

\begin{proof}
Let $[\alpha]$ be a class in $E^0(BH)$ represented by a map of spectra $\alpha: \Sigma_+^\infty(BH) \to E$. After unpacking the definition of $P_m$, the propostion amounts to the commutativity of the following diagram, which is guaranteed by Proposition \ref{transfer_ext_powers}.

\begin{equation*}
	\begin{tikzcd}
			& E\\
			& \PP_m(E) \arrow{u}[swap]{\mu_m} \\
		\PP_m(\Sigma_+^\infty BG) \arrow{ur}{\PP_m(\transfer^G_H(\alpha))} \arrow{r}{\PP_m(\transfer^G_H)} & \PP_m(\Sigma_+^\infty BH) \arrow{u}[swap]{\PP_m	(\alpha)}\\
		\Sigma_+^\infty BG\wreath\Sigma_m \arrow{u}{\simeq} \arrow{r}{\transfer^{G\wr\Sigma_m}_{H\wr\Sigma_m}} &\Sigma_+^\infty BH\wreath\Sigma_m \arrow{u}{\simeq}
	\end{tikzcd}		
\end{equation*}
\end{proof}

\subsection{Bimodule Duality}
Several objects that will appear later have two natural $E_0$-module structures. Although the ring $E_0$ is commutative, we will want to regard these as $E_0$-$E_0$ bimodules, and we will want to be careful about taking $E_0$-linear duals. We briefly review some notions about duality in this setting. We found Chapter 4 of Ponto's book on bicategories \cite{Ponto_trace_bicategories} helpful when thinking about these things, and will refer there for the proofs in this section.

\begin{defn} 
An $R$-$S$ bimodule ${}_R M_S$ is \emph{right dualizable} if it is finitely generated and projective as a right $S$-module. Dually, it is \emph{left dualizable} if it is a finitely generated projective left $R$-module.
\end{defn}

\begin{defn}If ${}_R M_S$ is an $R$-$S$ bimodule, its \emph{right dual} ${}_S M^\star_R$ is the function module $\operatorname{Hom}_{\operatorname{Mod}_S}({}_R M_S, {}_S S_S)$, taken in the category of right $S$-modules. This inherits a right $R$-module structure from the left $R$-module structure on $M$ and a left $S$-module structure from the natural left $S$-module structure on $S$. We can similarly define the left dual of an $R$-$S$-bimodule. It is also an $S$-$R$-bimodule.
\end{defn}

Many of the usual statements about duals and dualizability hold in this context. We collect a few for future use.

\begin{prop}
If ${}_R M_S$ is a right dualizable $R$-$S$ bimodule, then ${}_S M^\star_R$ is left dualizable, and its left dual is naturally isomorphic to ${}_R M_S$.
\end{prop}
\begin{proof}
This is essentially part of Ponto's definition of dualizability  (\cite{Ponto_trace_bicategories}, 4.3.1).  Proposition 4.2.1 there gives the equivalence between the two definitions.
\end{proof}

\begin{prop}
If a bimodule is right dualizable, any two right duals are naturally isomorphic.
\end{prop}

\begin{proof}
Ponto proves this between Propositions 4.3.2 and 4.3.3.
\end{proof}

Note that the right dualizability of a bimodule does not actually depend at all on the left module structure. 

\begin{prop}\label{dual_of_tensor}
If ${}_R M_S$ and ${}_S N_T$ are right dualizable bimodules with right duals ${}_S (M^\star)_R$ and ${}_T (N^\star)_S$, then ${}_R M_S\otimes {}_S N_T$ is also right dualizable, with right dual ${}_T (N^\star)_S \otimes {}_S (M^\star)_R$
\end{prop}
\begin{proof}
This is a special case of Ponto's Theorem 4.3.4.
\end{proof}

For the rest of this paper, all bimodules will be $E_0$-$E_0$-bimodules. We will also make a change of notation. If $f: E_0 \to \operatorname{Hom}_{\mathbb{Z}}(M, M)$ is a map of abelian groups which endows $M$ with either a left or right $E_0$-module structure, we write a superscript $f$ on the appropriate side of $M$. For example, if $M$ is a left $E_0$-module via a map $f$, and a right $E_0$-module via $g$, we indicate the resulting bimodule by ${}^f M^g.$

\section{Cohomology of Wreath Products}\label{cohom_wreath_products}
Since this paper is about the Morava $E$-cohomology of (the classifying spaces of) several groups, we use this section to collect some facts and definitions in that area.

The first is that the Morava $E$-theory of groups is often free, and if it is, we can identify the dual.
\begin{example}\label{cohom_of_sym_gps}
Let $G$ be a finite group. Some results from \cite{strickland_symmetric} (Corollary 3.8, Proposition 3.5 and the argument of Proposition 3.6) conspire to show that if $K(h)^*(BG)$ is concentrated in even degrees,  then $E^*(BG)$ is a finite free $E_*$-module which is concentrated in even degrees. Further, its $E_*$-linear dual is $E_*^\wedge(BG)$ and both of these statements remain true replacing $E_*$ with $E_0$. Theorem E of \cite{HKR} shows that this $K(h)$-hypothesis holds for any group built from symmetric groups by successively taking products and wreath products, which covers every group appearing in this paper. In particular, $E^0(B\Sigma_m)$ is finite free, and its $E_0$-linear dual is $E_0^\wedge(B\Sigma_m)$.
\end{example}

We can also define a bimodule structure on a certain quotient of $E^0(B\Sigma_m)$.

\begin{example}\label{quotient_of_sym_gp_bimod}
Strickland \cite{strickland_symmetric} considers an ideal $I_m$ in $E^0(B\Sigma_m)$ generated by the images of the transfer maps from the cohomology groups $E^0(B(\Sigma_i\times\Sigma_{m-i}))$. Proposition \ref{total_power_of_sum} shows that the total power operation $P_m: E_0 \to E^0(B\Sigma_m)$ becomes a ring map (but \emph{not} an $E_0$-algebra map) after composition with the quotient map $E^0(B\Sigma_m) \to E^0(B\Sigma_m)/I_m.$ Together with the usual multiplication on $E^0(B\Sigma_m)/I_m$, this ring map defines a second $E_0$-module structure on $E^0(B\Sigma_m)/I_m$, which we will call $t$ and regard as a left module structure. Using the letter $s$ for the usual module structure on this quotient, we get the $E_0$-$E_0$-bimodules ${}^t(E^0(B\Sigma_m)/I_m)^s$ appearing in the main theorem.
\end{example}

There is an analog of this example for the iterated wreath products of symmetric groups.

\begin{example}\label{quotient_of_wreath_product_bimod}
Let $\mathbf{m} = (m_1, \ldots, m_\ell)$ be a sequence of positive integers, and let $\Sigma_{\mathbf{m}} = \Sigma_{m_1}\wr\ldots\wr\Sigma_{m_\ell} = (\Sigma_{m_1}\wr\ldots\wr\Sigma_{m_{\ell-1}})\wr\Sigma_{m_\ell}$ be an iterated wreath product of symmetric groups. Define proper subgroups $\Sigma_{\mathbf{m},t,j}$ of $\Sigma_{\mathbf{m}}$ for $1\leq t \leq \ell$ and $0<j<m_t$ by 

\[\Sigma_{\mathbf{m},t,j} = \Sigma_{m_1}\!\wr\ldots\wr\!(\Sigma_j\times \Sigma_{m_t-j})\!\wr\ldots\wr\!\Sigma_{m_\ell}.\] 

In words, $\Sigma_{\mathbf{m},t,j}$ is an $\ell$-fold wreath product like $\Sigma_\mathbf{m}$, but with $\Sigma_j\times \Sigma_{m_t-j}$ inserted in the $t$-th slot. When $\mathbf{m}$ has length one, this recovers $\Sigma_i \times \Sigma_j$ as a subgroup of $\Sigma_m$ (where $i+j=m$). If $\mathbf{m} = (m_1,m_2)$ has length two, we have that $\Sigma_\mathbf{m} = \Sigma_{m_1}\wr\Sigma_{m_2},$ and the $\Sigma_{\mathbf{m},t,j}$ run over the subgroups $(\Sigma_i \times \Sigma_{m_1-i})\wr\Sigma_m$ and $\Sigma_{m_1}\wr(\Sigma_k\times \Sigma_{m_2-k})$, for $0< i < m_1$ and $0<k<m_2$. 

We now get transfer maps in cohomology $\transfer^{\Sigma_\mathbf{m}}_{\Sigma_{\mathbf{m},t,j}}: E^0(B\Sigma_{\mathbf{m},t,j}) \to E^0(B\Sigma_\mathbf{m})$ and can define a transfer ideal in $E^0(B\Sigma_{\mathbf{m}})$, like in Example \ref{quotient_of_sym_gp_bimod}

\begin{defn}\label{wreath_product_transfer_ideal}
The \emph{transfer ideal} $I_\mathbf{m}$ in $E^0(B\Sigma_\mathbf{m})$ is the ideal generated by the images of all the transfer maps $\transfer^{\Sigma_\mathbf{m}}_{\Sigma_{\mathbf{m},t,j}}$, for $1\leq t \leq \ell$ and $0<j<m_t$.
\end{defn}

Intuitively, $I_\mathbf{m}$ is the ideal generated by images of transfers from partition subgroups taken in only one wreath product variable at a time. When $\mathbf{m}$ has length one, we recover Strickland's ideals $I_m$.

Given $\mathbf{m}$, we also use $\mathbf{m}' = (m_1, \ldots, m_{\ell-1})$, which is $\mathbf{m}$ with the last entry deleted. Note that $\Sigma_{\mathbf{m}'} = \Sigma_{m_1}\wr\ldots\wr\Sigma_{m_{\ell-1}}$.

Now we can consider a composite $P_\mathbf{m}$ of total power operations as follows:

\[ E^0(\operatorname{pt}) \xrightarrow{P_{m_1}} E^0(B\Sigma_{m_1}) \xrightarrow{P_{m_2}} E^0(B\Sigma_{m_1}\wr\Sigma_{m_2})\xrightarrow{P_{m_3}} \cdots  \xrightarrow{P_{m_{\ell}}} E^0(B\Sigma_{\mathbf{m}}).
\]

By Proposition \ref{transfers_and_total_power_ops}, the map $P_{m_\ell}$ descends to a map $\bar{P}_{m_\ell}$ on quotients, making the following diagram commute.
\begin{equation*}
\begin{tikzcd}
	E^0(B\Sigma_{\mathbf{m}'}) \arrow{r}{P_{m_\ell}} \arrow{d} & E^0(B\Sigma_\mathbf{m}) \arrow{d} \\
	\frac{E^0(B\Sigma_{\mathbf{m}'})}{I_{\mathbf{m}'}} \arrow{r}{\overline{P}_{m_\ell}} & \frac{E^0(B\Sigma_{\mathbf{m}})}{I_{\mathbf{m}}}
\end{tikzcd}
\end{equation*}
The map $\bar{P}_{m_\ell}$ is a ring map (since in the target we have also killed transfers in the last wreath factor), so the composite $\bar{P}_\mathbf{m} = \bar{P}_{m_\ell} \circ \ldots \circ\bar{P}_{m_1}$ gives a ring map 
\[ E^0(\operatorname{pt}) \to \frac{E^0(B\Sigma_\mathbf{m})}{I_\mathbf{m}}.
\]
As in Example \ref{quotient_of_sym_gp_bimod}, this gives a left $E_0$-module structure on $E^0(B\Sigma_\mathbf{m})/I_\mathbf{m}$, which we will also denote with a $t$. This is defined in such a way as to make the map 
\[\biggl.^t\!\left(\frac{E^0(B\Sigma_{\mathbf{m}'})}{I_{\mathbf{m}'}}\right)\!\biggr.^s \xrightarrow{\bar{P}_{m_\ell}}  \biggl.^t\!\left(\frac{E^0(B\Sigma_{\mathbf{m}})}{I_{\mathbf{m}}}\right)\!\biggr.^s
\]
into a map of left $E_0$-modules.

In the course of the proof of the main theorem (Theorem \ref{main_theorem}), we will show that ${}^t(E^0(B\Sigma_{\mathbf{m}})/I_{\mathbf{m}})^s$ is right dualizable. We will define its right dual in Section \ref{J_bold_m}, but the proof that these are dual will also need to wait until we prove the main theorem.

\end{example}

\section{Rezk's Algebraic Approximation Functors}\label{the_functors_TT}

To study Morava $E$-theory power operations, Rezk (\cite{congruence_criterion}) introduced a sequence of functors $\TT_m$. These are an algebraic approximation to the $K(h)$-localized extended power functors. We will not worry much about the specifics of the construction of these functors and merely use some of their properties. We give a brief review of the necessary facts here.

\begin{theorem}[Rezk]\label{Tm_exists}
There exist functors $\TT_m$ from the category of $E_0$-modules to itself, together with natural transformations
\[\alpha_m: \TT_m(\pi_0) \to \pi_0L_{K(h)}\PP_m\]

between functors from the category of $E$-modules to the category of $E_0$-modules. This natural transformation is an isomorphism on finite free $E$-modules. The functors $\TT_m$ also preserve filtered colimits.
\end{theorem}

\begin{remark}
Rezk does not actually define the $\TT_m$ as endofunctors of the category of $E_0$-modules, but rather as endofunctors of the category of graded $E_*$-modules. Since $E_*$ is $2$-periodic and concentrated in even degrees, we can consider an $E_0$-module $M$ as a graded $E_*$-module by inserting $M$ in all even degrees, and  0 in all odd degrees. We can then apply Rezk's $\TT_m$ and take the degree 0 piece to recover what we call $\TT_m$ here.
\end{remark}

There are a few properties of the functors $\TT_m$ that we need aside from the ones stated in the theorem. We collect them in the following three propositions.

\begin{prop}\label{Tm_of_sum}
If $M$ and $N$ are $E_0$-modules, then we have an isomorphism
\[\TT_m(M \oplus N) \cong \bigoplus_{i=0}^m \TT_i(M) \otimes \TT_{m-i}(N).\]

More generally, if $M_0, \ldots , M_r$ are $r+1$ $E_0$-modules, there are isomorphisms  
\[\TT_m\left(\bigoplus_{i=0}^r M_i\right) \cong \bigoplus_{\lambda\in \Lambda_{m,r}} \bigotimes_{i=0}^r \TT_{\lambda_i}(M_i).\]

If $M_0$ and $N_{i,0}$ ($0\leq i \leq n$) are  $E_0$-modules, and $f_i:M_0\to N_{i,0}$ are $E_0$-module maps, then the isomorphisms above given an identification
\[\TT_m\left(\bigoplus_{i=0}^n f_i\right) \cong \bigoplus_{\lambda\in \Lambda_{m,n}}\bigotimes_{i=0}^n \TT_{\lambda_i}(f_i).\]
\end{prop}

\begin{proof}
These all follow from Proposition 4.14 of \cite{congruence_criterion}.
\end{proof}

\begin{prop}\label{Tm_of_point}
\[\TT_m(E_0) = E^{\wedge}_0(B\Sigma_m).\] 
\end{prop} 

\begin{proof}
This is more or less by the definition of all the symbols involved. Indeed, since $E_0$ is tautologically free as an $E_0$-module,
\begin{align*}
E^{\wedge}_0(B\Sigma_m) &= \pi_0L_{K(h)}(E\wedge \Sigma^\infty_+ B\Sigma_m) \\
		&= \pi_0L_{K(h)}\left((E\Sigma_m)_+ \wedge E^{\wedge_E m})_{\Sigma_m}\right) \\
		&= \pi_0L_{K(h)}\left((E^{\wedge_E m})_{h\Sigma_m}\right)\\
		&= \pi_0L_{K(h)}\PP^E_m(E) \\
		&= \TT_m(E_0).
\end{align*}
\end{proof}

\begin{prop}\label{comp_of_Tm}
If $G$ is a finite group so that $E^\wedge_0(BG)$ is a finite free $E_0$ module, then 
\[\TT_m\left(E^\wedge_0(BG)\right) = E_0^\wedge(BG\wr\Sigma_m).\]
In particular, 
\[\TT_m\left(\TT_n(E_0)\right) = E^{\wedge}_0(B\Sigma_n \wreath \Sigma_m)\]
and more generally, for any $k$-tuple $(m_1, \ldots m_k)$ of positive integers,
\[\TT_{m_k}\circ \ldots \circ \TT_{m_1}(E_0) = E_0^\wedge(B\Sigma_{m_1}\wr\ldots\wr\Sigma_{m_k}). \]
\end{prop}

\begin{proof}
We can make the following identifications.
\begin{align*}
\TT_m(E^\wedge_0(BG)) &= \pi_0L_{K(h)}\PP_m(L_{K(h)}(E\wedge \Sigma_+^\infty BG))\\
					 	&=\pi_0L_{K(h)}\PP_m(E\wedge \Sigma_+^\infty BG)\\
					 	&=\pi_0L_{K(h)}(E\wedge \Sigma_+^\infty BG\wr\Sigma_m)\\
					 	&=E_0^\wedge(BG\wr\Sigma_m).
\end{align*}
The first equality holds by the finite freeness assumption and Theorem \ref{Tm_exists}, the second is Proposition \ref{power_of_localization} and the third follows from Propositions \ref{extended_power_of_unit} and \ref{comp_P}. The last equality is Definition \ref{completed_homology}.
The first statement of the proposition applies to wreath products of symmetric groups by Example \ref{cohom_of_sym_gps}, and induction yields the second statement.
\end{proof}

\section{Left linear approximation and the chain rule}\label{left_linearization_section}
Consider two abelian categories $\cA$ and $\cB$, and let $F$ be a reduced, but not necessarily additive, functor from $\cA$ to $\cB$. In this setting, Rezk (\cite{power_ops_koszul}, 5.1) defines  a functor $\cL_F : \cA \to \cB$, which he calls the linearization of $F$. It serves as a universal additive functor $\cA \to \cB$ with a natural transformation from $F$. For our purposes it is more convenient to consider a version of linear approximation with the opposite universal property. We call Rezk's construction the linear approximation from the right, and now define the opposite notion. 
\begin{defn}\label{left_lin}
Let $\cA$ and $\cB$ be abelian categories and let $F: \cA \to \cB$ be which is not necessarily additive, but with the property $F(0) = 0$.
The \emph{left linearization} $\leftlin F$ of $F$ is the equalizer of 
\[\begin{tikzcd}[column sep=large]
	F(X)\arrow[yshift=0.7ex]{r}{F(i_1 + i_2)} \arrow[yshift=-0.7ex]{r}[swap]{F(i_1) + F(i_2)} & F(X\oplus X).
 \end{tikzcd}
\]
\end{defn}

The universal property of an equalizer provides a natural transformation $\eta: \leftlin F \to F$, and we have the following proposition.

\begin{prop}\label{univ_prop_left_lin}
$\leftlin F$ is additive, and any natural map $G \to F$ from an additive functor $G$ factors uniquely through $\eta: \leftlin F \to F$.
\end{prop}

\begin{proof}
This is formally dual to the proof of proposition 5.2 in \cite{power_ops_koszul}. Note that composition of functions in an abelian category is additive in both variables. Essentially, $i_1 + i_2$ is the universal case, as $f + g = (f \coprod g)\circ (i_1 + i_2): X \to X\oplus X \to Y$.
\end{proof}

Rezk then gives a certain sort of chain rule for right linearization in certain circumstances. The goal for the rest of this section is to  prove the analogous theorem for left linearization. First make the following definitions.
\begin{defn}\label{cross_effects_monad}
Let $\top F$ be a functor defined as 
\[\top F(X) := \operatorname{coker}\left( F(X)\oplus F(X) \xrightarrow{F(i_1)\coprod F(i_2)} F(X\oplus X)\right).\]
Write $\beta^F: F(X\oplus X) \to \top F(X)$ for the quotient map, and $\ell^F$ for $\beta^F \circ F(i_1 + i_2).$  
\end{defn}

There is now an evident exact sequence (natural in $X$)
\[0 \to \leftlin F(X) \xrightarrow{\eta} F(X) \xrightarrow{\ell^F} \top F(X).\]

Since $\leftlin F \circ \leftlin G$ is additive, there is a natural map $c:\leftlin F \circ \leftlin G \to \leftlin (F\circ G)$ given by Proposition \ref{univ_prop_left_lin}.
The chain rule in this context says that $c$ is an isomorphism when the above short exact sequence for $G(X)$ splits.

\begin{theorem}[Chain Rule]\label{Dual_chain_rule}
Let $\mathcal{A},\mathcal{B}$ and $\mathcal{C}$ be abelian categories, let $G:\mathcal{A} \to \mathcal{B}$ and $F:\mathcal{B}\to \mathcal{C}$ be reduced but not necessarily additive functors between them, and let $X$ be an object of $\mathcal{A}$. Suppose there are splittings $\sigma: G(X) \to \leftlin G(X)$ and $\tau: \top G(X) \to G(X)$, so that $\sigma\circ\eta = \operatorname{id}_A$ and $\tau\circ\ell^G|_{\operatorname{ker}\sigma} = \operatorname{id}_{\operatorname{ker}\sigma}$. Then the comparison map $c_X: (\leftlin F\circ \leftlin G)(X) \to \leftlin(F\circ G)(X)$ is an isomorphism.
\end{theorem}

\begin{proof}
  Rezk gives two proofs of his version: A brief proof relying on work of Johnson and McCarthy (\cite{johnson_mccarthy}) and an elementary but tedious proof that amounts to unravelling the first. Dualizing either proof gives our version of the theorem. 
\end{proof}

The left linearizations of the functors $\TT_m$ behave particularly nicely.

\begin{prop}\label{dual_lin_colimits}
The functors $\leftlin \TT_m$ preserve filtered colimits and arbitrary (small) coproducts.
\end{prop}
\begin{proof}
Theorem \ref{Tm_exists} includes the fact that the $\TT_m$ preserve filtered colimits. Since $\leftlin \TT_m$ is defined as a finite limit, this gives that $\leftlin \TT_m$ commutes with filtered colimits. Of course, $\TT_m$ does not preserve finite coproducts, but the left linearization of a reduced functor does by construction. An arbitrary small coproduct can be written as a filtered colimit of finite coproducts, so we get the result.
\end{proof}

\begin{prop}\label{right_mod_struct_on_linearization}
The $E_0$-module $\leftlin \TT_m(E_0)$ inherits an additional right $E_0$-module structure from the usual right $E_0$-module structure on $E_0$.
\end{prop}

\begin{proof}
We can identify $E_0$ with $\Mod_{E_0}(E_0, E_0)$ by sending an element to multiplication on the right by that element. The functor $\leftlin\TT_m$ is additive, so we can define a map 
\[t: E_0 = \Mod_{E_0}(E_0, E_0) \to \Mod_{E_0}(\leftlin\TT_m(E_0), \leftlin\TT_m(E_0))\]
by the formula $t:f\mapsto \leftlin\TT_m(f).$ This defines our desired module structure.
\end{proof}

\begin{remark}\label{Tm_bimodule_labelling}
We also use the letter $s$ for the original (left) module structure on $\leftlin\TT_m$. This is intentionally the opposite choice of letters and sides as $E^0(B\Sigma_m)/I_m$. It is also the same sort of bimodule structure that appears in the Eilenberg-Watts Theorem.
\end{remark}

This determines the values of $\leftlin \TT_m$ on all finitely generated projective modules.

\begin{prop}\label{ff_comp}
Let $M$ be a finitely generated projective left $E_0$-module concentrated in even degrees. Then there is a natural isomorphism $\leftlin \TT_m(M) \leftarrow {}^s(\leftlin \TT_m(E_0))^t \otimes_{E_0} M$.
\end{prop}

\begin{proof}
The functor $\leftlin \TT_m$ is additive by construction, so we can ask for its right exact approximation $R_0\leftlin \TT_m$, which comes with a comparison map $R_0\leftlin \TT_m \to \leftlin \TT_m$. This map is always an isomorphism on projective modules. Since $\leftlin \TT_m$ preserves coproducts by Proposition \ref{dual_lin_colimits}, the Eilenberg-Watts theorem applies and says that $R_0\leftlin \TT_m$ is given by the cited tensor product.
\end{proof}

\section{Some Bimodules}\label{bimodules}
Though the statement of Theorem \ref{main_theorem} is about certain quotients of cohomology modules, the proof will make heavy use of a dual version. This section defines and studies those duals.

\subsection{The Bimodules $J_m$}

Consider the subgroups $\Sigma_i \times \Sigma_j$ of $\Sigma_m$, for $i+j=m$, $i,j\neq m$ from Example \ref{embeddings}. The inclusions induce transfer maps 

\[E_0^{\wedge}(\Sigma_m) \to E_0^{\wedge}(B(\Sigma_i\times \Sigma_j) \cong E_0^{\wedge}(B\Sigma_i)\otimes E_0^{\wedge}(B\Sigma_j).\] 

Taking the product of these transfers gives a map
\begin{equation}\label{sum_homology_transfer}
	E_0^{\wedge}(\Sigma_m) \to \bigoplus_{i= 1}^{m-1} E_0^{\wedge}(B(\Sigma_i\times \Sigma_{m-i}).
\end{equation} 

\begin{defn}\label{J_m}
	The (left) $E_0$ module $J_m$ is the kernel of the map  (\ref{sum_homology_transfer}). That is, $J_m$ is the common kernel of all transfers from the non-trivial partition subgroups.
\end{defn}

\begin{remark}
In the case where $n=p^k$, for some $k$, Rezk (\cite{congruence_criterion,power_ops_koszul}) calls these $\Gamma[k]$. If $n$ is not a power of $p$, $J_m = 0$. These are also not to be confused with the modules Strickland calls $J_m$ in \cite{strickland_symmetric}.
\end{remark}

\begin{prop}\label{J_m_duality}
The (left) $E_0$-module $J_m$ is finite free, and is $E_0$-linearly dual to the (right) $E_0$-module $E^0(B\Sigma_m)/I_m$
\end{prop}

\begin{proof}
Strickland \cite{strickland_symmetric} shows that transfer maps in (completed) $E$-homology are dual to transfer maps in $E$-cohomology, if the groups involved have free $E$-cohomology. Thus, the map in (\ref{sum_homology_transfer}) is dual to the cohomology map whose cokernel is $E^0(B\Sigma_m)/I_m$ (with the right module structure from before). Since  $E^0(B\Sigma_m)/I_m$ is finite free as a right module, we get the proposition.
\end{proof}

In Example \ref{quotient_of_sym_gp_bimod}, we saw that $E^0(B\Sigma_m)/I_m$ has a bimodule structure, coming from a total power operation. The previous proposition identifies $J_m$ as $\operatorname{Hom}_{E_0}(E^0(B\Sigma_m)/I_m, E_0)$. The left module structure on $E^0(B\Sigma_m)/I_m$ then induces a right module structure on $J_m$, but this .

\begin{defn}\label{right_mod_J_m}
The right module structure on $J_m$ is the dual of the left module structure $t$ on $E^0(B\Sigma_m)/I_m$. We will also refer to this right module structure as $t$.
\end{defn}


It turns out that $J_m$ also appears as a value of the left linearization of $\TT_m$.

\begin{prop}\label{dual_lin_on_point}
As a left $E_0$-modules, we can identify $\leftlin\TT_m(E_0)$ and ${}^s\!J_m.$
\end{prop}

\begin{proof}
We have that $\TT_m(E_0 \oplus E_0) = \pi_0L_{K(h)}\PP_m^E(E\vee E)$, and $\TT_m(\Delta) = \pi_0L_{K(h)}\PP_m^E(\Delta)$. By Propositions \ref{Tm_of_point} and \ref{extended_power_of_wedge}, we can make the identifications
\[\TT_m(E_0 \oplus E_0) \cong \bigoplus_{i+j = n} \TT_i(E_0) \otimes_{E_0} \TT_j(E_0) \cong \bigoplus_{i+j = n} E_0^{\wedge}(B(\Sigma_i\times \Sigma_j)).\]
Then by Proposition \ref{transfer_is_diagonal}, 
\[\TT_m(\Delta) = \bigoplus_{i+j=m} \transfer_{\Sigma_i \times \Sigma_j}^{\Sigma_m}.\]
The other map $\TT_m(i_1)+\TT_m(i_2)$ in the definition of left linearization is the inclusion into the $i=0$ and $j=0$ factors. Then the equalizer is the kernel of the nontrivial transfer maps
\[\leftlin\TT_m(E_0) = \operatorname{ker}\left(\bigoplus_{0<i<m}\transfer_{\Sigma_i \times \Sigma_{m-i}}^{\Sigma_m}\right).\]
By Definition \ref{J_m} this kernel is also $J_m$ (with the usual module structure $s$).
\end{proof}

Proposition \ref{right_mod_struct_on_linearization} gives $\leftlin \TT_m(E_0)$ a bimodule structure. The isomorphism of Proposition \ref{dual_lin_on_point} is in fact an isomorphism of bimodules.

\begin{prop}\label{module_structures_same}
As an $E_0$-$E_0$-bimodule, ${}^s\leftlin \TT_m(E_0)^t$ is isomorphic to ${}^s\!J_m^t$
\end{prop}

\begin{proof}

By Proposition \ref{extended_power_of_unit} the multiplication map $\eta :\PP_m^E E \to E$ is also given by the map crushing $B\Sigma_m$ down to a point.  Also, the diagonal map of $B\Sigma_m$ induces $\PP_m^E E$ a comultiplication \[\Delta: \PP_m^E E \to \PP_m^E E \sm_E \PP_m^E E,\]. These two maps give $\PP_m^E E$ a comonoid structure in the category of $E$-modules (with symmetric monoidal product $\wedge_E$). We then get a monoid structure $(\mu, \epsilon)$ on $\Mod_E(\PP_m^E E, E).$
The following diagram commutes, by the counit identity for comonoids 
\begin{equation*}
\begin{tikzcd}[column sep=-3em]
 	\Mod_E(\PP_m^E E, \PP_m^E E) \arrow[crossing over]{rr}{L_{K(h)}} \arrow{dr} \arrow[swap]{dd}{\eta}&& \Mod_E(L_{K(h)}\PP_m^E E, L_{K(h)}\PP_m^E E)\arrow[equal]{dd} \arrow[crossing over]{dl}{\eta}\\
 	&\Mod_E(L_{K(h)}\PP_m^E E, E) \arrow{dl}{\simeq}& \\
 	\Mod_E(\PP_m^E E, E) \arrow{rr}{\operatorname{Adj}(\mu)} && \Mod_E(\Mod_E(\PP_m^E E, E), \Mod_E(\PP_m^E E, E))
\end{tikzcd}
\end{equation*}

We can make the left vertical identification since $L_{K(h)} (E\sm B\Sigma_m)$ is $K(h)$-locally dualizable (\cite{hovey_strickland} Corollary 8.7, or \cite{strickland_symmetric}, Proposition 3.7).
Since $E^0(B\Sigma_m)$ and $E^{\wedge}_0(B\Sigma_m)$ are free $E_0$-modules, we can apply $\pi_0$ and use the equivalence of categories from Proposition \ref{finite_free_equivalence} to get a commutative diagram
\begin{equation*}
\begin{tikzcd}
	\pi_0\Mod_E(\PP^E_m E, \PP^E_m E) \arrow{r} \arrow{d} & \Mod_{E_0}(E^{\wedge}_0B\Sigma_m, E^{\wedge}_0B\Sigma_m) \\
	E^0B\Sigma_m \arrow{r} & \Mod_{E_0}(E^0B\Sigma_m, E^0B\Sigma_m). \arrow[equal]{u}
\end{tikzcd}
\end{equation*}
Now we can precompose with the map of \emph{sets} \[\pi_0\PP_m^E: E_0 = \pi_0\Mod_E(E, E) \to \pi_0\Mod_E(\PP^E_m E, \PP^E_m E))\] to get a diagram (\emph{of sets})
\begin{equation*}
\begin{tikzcd}
	E_0 \arrow{dr}{\pi_0\PP_m^E} \arrow[bend right]{ddr} \arrow[bend left=10]{rrd}\\
	&\pi_0\Mod_E(\PP^E_m E, \PP^E_m E) \arrow{r} \arrow{d} & \Mod_{E_0}(E^{\wedge}_0B\Sigma_m, E^{\wedge}_0B\Sigma_m) \\
	&E^0B\Sigma_m \arrow{r} & \Mod_{E_0}(E^0B\Sigma_m, E^0B\Sigma_m) \arrow[equal]{u}
\end{tikzcd}
\end{equation*}
The bottom arrow is now the map realizing $E^0(B\Sigma_m)$ as a module over itself via multiplication., and the curved arrows are defined as the compostions making the diagram commute. The discussion from Section \ref{power_ops} tells us that the left curved arrow is the total power operation $P_m$ on $E_0$, and Theorem \ref{Tm_exists} shows the top curved arrow is given by application of the functor $\TT_m$ to $E_0$. Neither of these maps are additive. However, further applying $\leftlin$ to the top curved arrow gives the $E_0$-module structure $t$ on $\leftlin\TT_m(E_0)$. The left curved arrow and bottom arrow give the left $E_0$ module structure $t$ on $E^0(B\Sigma_m)/I_m$ after passing to quotients. Dualizing this gives the right $E_0$-module structure $t$ on $J_m$ from Definition \ref{right_mod_J_m}. The commutativity of these diagrams show that these two module structures coincide.
\end{proof}

\subsection{The Modules $J_\mathbf{m}$}

We now define the objects that will be dual to the bimodules of Example \ref{quotient_of_wreath_product_bimod}. These are analogs of the $J_m$, for multiple wreath product factors.

Let $\mathbf{m} = (m_1, \ldots, m_\ell)$ be a sequence of positive integers as in Example \ref{quotient_of_wreath_product_bimod}, and let $\Sigma_{\mathbf{m}}$ and $\Sigma_{\mathbf{m},t,j}$ be the iterated wreath products and subgroups thereof defined there. Let $\TT_\mathbf{m} = \TT_{m_\ell}\circ \ldots \circ \TT_{m_2} \circ \TT_{m_1}$, and note that $\TT_\mathbf{m}(E_0) = E_0^\wedge(B\Sigma_\mathbf{m}),$ by Proposition \ref{comp_of_Tm}. 

Now we can consider the common kernel of the homology transfer maps from these subgroups.

\begin{defn}\label{J_bold_m}
Set $J_\mathbf{m} \subseteq E_0^{\wedge}(B\Sigma_\mathbf{m})$ to be the kernel of the map given by
\[\bigoplus_{\substack{1\leq t\leq\ell\\
					0<j<m_t}}
\operatorname{tr}^{\Sigma_\mathbf{m}}_{\Sigma_{\mathbf{m},t,j}}.\] 
\end{defn}

For instance, if $\mathbf{m} = (m)$ has length one, we recover the $J_m$ of Definition \ref{J_m}.

\subsection{Identifying $J_\mathbf{m}$ with a left linearization}

For  $J_\mathbf{m}$ to be useful, we want to identify it as a value of a left linearization, as in Proposition \ref{dual_lin_on_point}. We need a couple of technical lemmas to let us do that.

Recall the sets $\Lambda_{m,k}$ from section \ref{extended_powers_section}:
\[
\Lambda_{m,k} = \left\{ (\lambda_0, \ldots, \lambda_m) \in \N^k \mid \sum_{i=0}^k \lambda_i = m \right\}.
\]
Let $G$ be a finite group, and let $\mathcal{H} = \{H_0, \ldots, H_k\}$  be a finite list of (not necessarily distinct) subgroups of $G$, with $H_0 = H_k = G$. Also assume that $E_0^\wedge(BG)$ and all of the $E_0^\wedge(BH_i)$ are finite free $E_0$-modules.

For $\lambda \in \Lambda_{m,k}$, define
\[
\widetilde{\Sigma_{\mathcal{H}, \lambda}} = \prod_{i=0}^k H_i\wreath\Sigma_{\lambda_i} \leq G\wreath\Sigma_m.
\]
Finally, consider the following two families of elements of $\Lambda_{m,k}$. For $0 \leq i \leq k$, define $\mu_i = (0, \ldots, 0, m, 0, \ldots, 0)$, where the only nonzero entry is in the $i$-th slot, and for $0 < r < m$, define $\nu_r = (r, 0, \ldots, 0, m-r)$. Note that 
\[
\widetilde{\Sigma_{\mathcal{H},\mu_i}} = H_i\wreath\Sigma_m,
\]
\[
\widetilde{\Sigma_{\mathcal{H},\nu_r}} = G\wreath(\Sigma_r\times\Sigma_{m-r}),
\]
and if $\lambda = (m, 0, \ldots , 0)$ or $\lambda = (0, \ldots, 0, m)$, then $\widetilde{\Sigma_{\mathcal{H}, \lambda}} =G\wreath\Sigma_m$.
\begin{lemma}\label{annoying_partition_lemma}
Unless $\lambda = \mu_i$ for some $i$, $\widetilde{\Sigma_{\mathcal{H},\lambda}}$ is a subgroup of $\widetilde{\Sigma_{\mathcal{H},\nu_r}}$ for some $0 < r < m.$
\end{lemma}

\begin{proof}
If $\lambda = (\lambda_0, \ldots , \lambda_k)$ is not a $\mu_i$, it has at least two nonzero entries. Let $r$ be the smallest integer such that $\lambda_r$ is nonzero. 
Then $\lambda_r \neq m,$ so $\widetilde{\Sigma_{\mathcal{H},\lambda}} \leq \widetilde{\Sigma_{\mathcal{H},\nu_{\lambda_r}}}.$
\end{proof}

The following lemma allows us to calculate the left linearization of iterated $\TT$'s. Recall the groups $\Sigma_{\mathbf{m},t,j}$ and $\Sigma_\mathbf{m}$ from Example \ref{quotient_of_wreath_product_bimod}

\begin{lemma} \label{subgroup_family_lemma}
For each $\Sigma_\mathbf{m}$, there is a collection $\mathcal{H}_\mathbf{m}$ of proper subgroups with the following properties:
\begin{enumerate}
\item Each $H$ in $\mathcal{H}_\mathbf{m}$ has finite free $E_0^\wedge(BH)$.
\item Each $\Sigma_{\mathbf{m},t,j}$ is in $\mathcal{H}_\mathbf{m}$.
\item Each $H$ in $\mathcal{H}_\mathbf{m}$ is contained in some $\Sigma_{\mathbf{m},t,j}$.
\item 
\[\TT_\mathbf{m}(E_0 \oplus E_0) \cong \TT_\mathbf{m}(E_0) \oplus \left(\bigoplus_{H \in \mathcal{H}_\mathbf{m}} E_0^{\wedge}(BH) \right)\oplus \TT_\mathbf{m}(E_0).\]
\item Under the identification of item (4), $\TT_\mathbf{m}(\Delta): \TT_\mathbf{m}(E_0) \to \TT_\mathbf{m}(E_0 \oplus E_0)$ is given by 
\[\TT_\mathbf{m}(i_1) \times \left(\prod_{H\in \mathcal{H}_\mathbf{m}} \operatorname{tr}_H^{\Sigma_\mathbf{m}}\right) \times \TT_\mathbf{m}(i_2).\]
\end{enumerate}
\end{lemma}

\begin{proof}
We'll induct on the length of $\mathbf{m}$. If the length is $1$ (so that $\mathbf{m} = (m)$), we can take $\mathcal{H}_\mathbf{m} = \{\Sigma_i\times \Sigma_{m-i}\}$ for $0<i<m$.

Now consider an arbitrary $\mathbf{m}$, and assume the lemma for $\mathbf{m}' = (m_1,\ldots,m_{\ell-1})$. 
Enumerate $\mathcal{H}_{\mathbf{m}'}$ as $(H_1, \ldots, H_{r-1}),$ and set $H_0 = H_r = \Sigma_{\mathbf{m}'}.$ 

We claim that the collection
\begin{align*}
\mathcal{H}_\mathbf{m} &= \left\{\prod_{i=0}^r H_i\!\wr\!\Sigma_{\lambda_i}\Big| \lambda \in \Lambda_{m_\ell,r}, \lambda\neq\mu_0,\mu_m\right\}\\
	&=\left\{ \widetilde{\Sigma_{\mathcal{H}_{\mathbf{m}'},\lambda}} \Big| \lambda \in \Lambda_{m_\ell,r}, \lambda\neq\mu_0,\mu_m\right\}\\
\end{align*}
will satisfy properties (1)-(5).

Each of the $H_i$ have finite free $E_0^\wedge(BH)$ by assumption. Taking products and wreathing with symmetric groups both preserve finite freeness in completed homology, so property (1) holds.

To check property (2), there are two cases. If $t=\ell$, we can take $\lambda = \nu_j$ to get $\Sigma_{\mathbf{m},\ell,j}$. 
 If $t<\ell$, note that $\Sigma_{\mathbf{m},t,j}= \Sigma_{\mathbf{m}',t,j}\wr\Sigma_{m_\ell}$. Since $\Sigma_{\mathbf{m}',t,j} \in \mathcal{H}_{\mathbf{m}'}$, we can take $\lambda$ to be $\mu_t$. 

Lemma \ref{annoying_partition_lemma} shows property (3), and for property (4) we can make the following computation.

\begin{align*}
		\TT_{m_\ell}(\TT_{\mathbf{m}'}(E_0 \oplus E_0)) 	&\cong \TT_{m_\ell}\left(\bigoplus_{i=0}^r E^{\wedge}_0(BH_i)\right)\\
										&\cong \bigoplus_{\lambda\in \Lambda_{m_\ell,r}} \bigotimes_{i=0}^r \TT_{\lambda_i} \left(E_0^{\wedge}BH_i\right)\\
										&\cong \bigoplus_{\lambda\in \Lambda_{m_\ell,r}} \bigotimes_{i=0}^r E_0^{\wedge} B\big(H_i\!\wr\!		\Sigma_{\lambda_i}\big)\\
										&\cong \bigoplus_{\lambda\in \Lambda_{m_\ell,r}} E_0^{\wedge} \Bigg( \prod_{i=0}^r H_i\!\wr\!\Sigma_{\lambda_i}\Bigg).
	\end{align*}
	The first isomorphism here follows from Proposition \ref{extended_power_of_wedge}, the second from Proposition \ref{Tm_of_point} and the fourth from Proposition \ref{comp_of_Tm}.

Finally, property (5) follows from combining Propositions \ref{transfer_ext_powers} and \ref{Tm_of_sum} with the inductive hypotheses.

\end{proof}

Properties (4) and (5) of Lemma \ref{subgroup_family_lemma} give us a way to compute $\leftlin(\TT_\mathbf{m})(E_0).$

\begin{cor}\label{leftlin_of_TT_bold_m}
\[
\leftlin(\TT_\mathbf{m})(E_0) = \operatorname{ker}\left( 
	\prod_{H\in \mathcal{H}_\mathbf{m}} 
	\operatorname{tr}_H^{\Sigma_\mathbf{m}}
	\right).
\]
\end{cor}

Finally, we can make our desired identification.
\begin{cor}\label{wreath_computation}
There is an isomorphism of left $E_0$-modules ${}^s\leftlin(\TT_\mathbf{m})(E_0) \cong {}^sJ_{\mathbf{m}}\subseteq E_0(B\Sigma_{m_1}\wreath\ldots\wr\!\Sigma_{m_\ell}).$ 
\end{cor}
\begin{proof}
Combine Corollary \ref{leftlin_of_TT_bold_m} and Properties (2) and (3) from Lemma \ref{subgroup_family_lemma}.
\end{proof}

\subsection{A Bimodule Structure on $J_{\mathbf{m}}$.}

The left $E_0$-modules $\leftlin(\TT_\mathbf{m})(E_0)$ also carry right $E_0$ module structures; the proof is identical to that of Proposition \ref{right_mod_struct_on_linearization}. We can now use this to give the $J_\mathbf{m}$ a bimodule structure.

\begin{defn}
The identification in Corollary \ref{wreath_computation} endows $J_\mathbf{m}$ with a right $E_0$-module structure $t$, coming from the right module structure $t$ on $\leftlin(\TT_\mathbf{m})(E_0)$.
\end{defn}

\begin{remark}
The definition of a bimodule structure on $J_m$ (that is, if $\mathbf{m}$ has length 1) was as the dual of the bimodule structure on $E^0(B\Sigma_m)/I_m$. Proposition \ref{module_structures_same} shows that there is no ambiguity in this case. One might like to similarly define the module structure on $J_\mathbf{m}$ as dual to that on $E^0(B\Sigma_\mathbf{m})/I_\mathbf{m}$. Unfortunately, we do not yet know if either of these objects are dualizable, let alone dual to one another. This is in fact the case, and a large portion of the proof of the main theorem is proving this. However, this seems to be inextricably tied up with the rest of that proof, so we cannot separate it out here.
\end{remark}

\section{Proof of the main theorem}\label{proof_of_main}

We're now ready to prove the theorem. For clarity, we restate it here.

\begin{theorem}\label{main_theorem}

Let $\mathbf{m} = (m_1, \ldots, m_\ell)$ be an $\ell$-tuple of positive integers, and let $\Sigma_\mathbf{m} = \Sigma_{m_1}\wr\ldots\wr\Sigma_{m_\ell}$ be an iterated wreath product of symmetric groups. There are ideals $I_\mathbf{m}$ in $E^0(\Sigma_\mathbf{m})$ analogous to  Strickland's ideals. The quotient $E^0(\Sigma_\mathbf{m})/I_\mathbf{m}$ has an $E_0$-module structure $t$ in addition to the usual module structure $s$.
Then there is an isomorphism of  $E_0$-$E_0$ bimodules
\[\biggl.^t\!\left(\frac{E^0(B\Sigma_{m_1} \wreath \ldots \wreath \Sigma_{m_\ell})}{I_{\mathbf{m}}}\right)\!\biggr.^s \cong \biggl.^t\!\left(\frac{E^0(B\Sigma_{m_1})}{I_{m_1}}\right)\!\biggr.^s \bigotimes \cdots \bigotimes \biggl.^t\!\left(\frac{E^0(B\Sigma_{m_\ell})}{I_{m_\ell}}\right)\!\biggr.^s.\]

	The right hand side is the tensor product of (right) dualizable $E^0$-modules, and so is a (right) dualizable $E^0$-module. In particular, both sides are finitely generated projective (and so, free) modules.
\end{theorem}

\begin{proof}
The ideals $I_\mathbf{m}$ were defined in Example \ref{quotient_of_wreath_product_bimod}, and the second module structure was constructed in that same example.
To prove the isomorphism, we induct on the length of $\mathbf{m}$. The statement is a tautology for length one sequences, but we do note that freeness in this case was proven by Strickland (\cite{strickland_symmetric}, Theorem 8.6.
Now consider a general $\mathbf{m} = (m_1, \ldots , m_\ell)$, and let $\mathbf{m}' = (m_1, \ldots, m_{\ell-1})$. First we dualize and prove an isomorphism 

\begin{equation}\label{dual_theorem}
	{}^s\!J_{\mathbf{m}}^t \simeq {}^s\!J_{m_\ell}^t \otimes {}^s\!J_{\mathbf{m'}}^t
\end{equation}

By Proposition \ref{wreath_computation}, the left hand side is $\leftlin(\TT_\mathbf{m})(E_0)$, and assembling Propositions \ref{wreath_computation}, \ref{ff_comp}, \ref{dual_lin_on_point} and \ref{module_structures_same} gives us that the right hand side is $\leftlin(\TT_{\mathbf{m}'}) \circ \leftlin(\TT_{m_\ell})(E_0)$. We now need to verify the hypotheses of the chain rule (Theorem \ref{Dual_chain_rule}) and it will give our isomorphism. Specifically, we must show that the sequence

\begin{equation*}
	0 \to \leftlin(\TT_{m_\ell})(E_0) \to \TT_{m_\ell}(E_0) \to \top\TT_{m_\ell}(E_0)
\end{equation*}
splits.

Proposition \ref{dual_lin_on_point} allows us to identity this sequence as 
\begin{equation*}
	0 \to J_{m_\ell} \to E_0^{\wedge}(B\Sigma_{m_\ell}) \to \bigoplus_{i=1}^{n_\ell -1} E_0^{\wedge}(B(\Sigma_i \times \Sigma_{m_\ell-i}))
\end{equation*}

where the second map is the sum over appropriate transfers.

By Examples \ref{cohom_of_sym_gps} and \ref{quotient_of_sym_gp_bimod}, all of the terms in this sequence are free, and taking $E^0$-linear duals, we get the sequence

\begin{equation*}
 \bigoplus_{i=1}^{n_\ell -1} E^{0}(B(\Sigma_i \times \Sigma_{m_\ell-i})) \to E^0(B\Sigma_{m_\ell}) \to \frac{E^0(B\Sigma_{m_\ell})}{I_{m_\ell}} \to 0.
\end{equation*}

Since all of the terms in this sequence are free  we get the duals of the required splittings, and dualizing back gives the splittings for the original sequence.

Now we've established an isomorphism as in (\ref{dual_theorem}). It remains to show that we can in fact dualize again and recover the isomorphism of the theorem. The right hand side of the isomorphism in the statement of the theorem is a right dualizable $E_0$-module by induction and Proposition \ref{dual_of_tensor}; its dual is the right hand side of the isormorphism (\ref{dual_theorem}). That isomorphism then gives that ${}^s\!J_{\mathbf{m}}^t$ is finite free as a left module. To finish the proof we need to show that $E^0(B\Sigma_\mathbf{m})/I_\mathbf{m}$ is finite free as a right module, with right dual ${}^s\!J_{\mathbf{m}}^t$.

To prove that, consider the two sequences

\[
		\bigoplus E^0(\Sigma_{\mathbf{m},t,j}) \xrightarrow{\oplus \transfer_{\Sigma_{\mathbf{m},t,j}}^{\Sigma_{\mathbf{m}}}}  E^0(B\Sigma_{\mathbf{m}}) \to \frac{E^0(B\Sigma_\mathbf{m})}{I_\mathbf{m}} \to 0
\]

and 

\begin{equation*}
	\Hom\left(\bigoplus E_0^{\wedge}(\Sigma_{\mathbf{m},t,j}), E_0\right) \to \Hom(E_0^{\wedge}B\Sigma_\mathbf{m}, E_0) \to \Hom(J_\mathbf{m}, E_0) \to 0.
\end{equation*}

These are exact by Lemma \ref{subgroup_family_lemma} and freeness (for the second sequence). The first terms of these sequences are canonically isomorphic by duality, as are the second terms. Under these isomorphisms, the first maps also get identified. Thus, the third terms are also isomorphic. Since the third term in the second sequence is finite free (as Hom out of a finite free module), the third term of the first sequence is also, and we are done.


\end{proof}
\bibliographystyle{amsalpha}
\bibliography{refs.bib}

\end{document}